\documentclass{article}
\usepackage{graphicx}
\usepackage[left=3cm, top=2cm, right=3cm, bottom=3cm]{geometry}
\usepackage{listings}
\usepackage{float}
\usepackage{amssymb,amsmath,mathtools}
\DeclarePairedDelimiter{\ceil}{\lceil}{\rceil}
\usepackage{enumerate}
\usepackage[numbers,square,sort]{natbib}
\usepackage{wrapfig}
\usepackage{framed}
\usepackage{tabstackengine}
\usepackage{braket}
\usepackage[nottoc]{tocbibind}\settocbibname{References}
\usepackage{mathrsfs}

\usepackage{indentfirst}
\usepackage{setspace}
\usepackage{bm}
\usepackage[dvipsnames]{xcolor}

\usepackage{tikz}
\tikzset{every loop/.style={}}
\usetikzlibrary{arrows,shapes,decorations.markings,automata,backgrounds,petri,bending,calc}

\usepackage{fancyhdr}
\fancypagestyle{firststyle}
{
	
	\fancyfoot[l]{\footnotesize\textit{Date:} \today}
	\fancyfoot[c]{1}
	\fancyhead{}

}

\usepackage{tikz-cd} 
\tikzset{
    labl/.style={anchor=south, rotate=90, inner sep=.5mm}
}

\usepackage{enumitem}
\setlist[enumerate,1]{label=(\arabic*)}
\newlist{steplist}{enumerate}{1}
\setlist[steplist]{label={Step \arabic*:}, ref={Step \arabic*}}

\usepackage[utf8]{inputenc}
\usepackage[english]{babel} 

\usepackage{amsthm}
\newtheorem{thm}{Theorem}[section]
\newtheorem{lem}[thm]{Lemma}
\newtheorem{prop}[thm]{Proposition}

\numberwithin{equation}{section}

\theoremstyle{definition}
\newtheorem{defn}[thm]{Definition} 
\newtheorem{remk}[thm]{Remark}

\newcommand{\cC}{\mathcal{C}}

\newcommand{\cH}{\mathcal{H}}

\newcommand{\cV}{\mathcal{V}}

\newcommand{\G}{\Gamma}
\newcommand{\La}{\Lambda}

\newcommand{\fg}{\mathfrak{g}}
\newcommand{\fF}{\mathfrak{F}}

\newcommand{\acts}{\curvearrowright}

\makeatletter
\newsavebox{\@brx}
\newcommand{\llangle}[1][]{\savebox{\@brx}{\(\m@th{#1\langle}\)}%
	\mathopen{\copy\@brx\kern-0.5\wd\@brx\usebox{\@brx}}}
\newcommand{\rrangle}[1][]{\savebox{\@brx}{\(\m@th{#1\rangle}\)}%
	\mathclose{\copy\@brx\kern-0.5\wd\@brx\usebox{\@brx}}}
\makeatother

\begin{document}
\begin{center}
{\LARGE\bf
A cubulation with no factor system}\\
\bigskip
{\large Sam Shepherd}
\end{center}

\begin{abstract}
	The primary method for showing that a given cubulated group is hierarchically hyperbolic is by constructing a factor system on the cube complex.
	In this paper we show that such a construction is not always possible, namely we construct a cubulated group for which the cube complex does not have a factor system.
	We also construct a cubulated group for which the induced action on the contact graph is not acylindrical.
\end{abstract}
\bigskip

\medskip
\section{Introduction}

A \emph{cubulated group} $G\acts X$ is a group $G$ together with a proper cocompact action of $G$ on a CAT(0) cube complex $X$ (and if $G$ is fixed then each such action is called a \emph{cubulation of $G$}).
Numerous groups can be cubulated, including small cancellation groups, finite volume hyperbolic 3-manifold groups and many Coxeter groups -- see \cite{WiseRiches} for further background and examples.
In turn, many cubulated groups are examples of hierarchically hyperbolic groups (HHG's), a class of groups that includes hyperbolic groups, relatively hyperbolic groups and mapping class groups among others -- see \cite{BehrstockHagenSisto17,BehrstockHagenSisto19} for relevant definitions and background.
The primary method for showing that a given cubulated group $G\acts X$ is a HHG is by constructing a certain family of subcomplexes of $X$, called a factor system, which we define below following \cite{BehrstockHagenSisto17}.
Many cubulated groups are known to have factor systems, including virtually special cubulated groups \cite[Proposition B]{BehrstockHagenSisto17} -- see also \cite{HagenSusse20}.

\begin{defn}\label{defn:factor}
	Let $X$ be a CAT(0) cube complex. Each hyperplane $H$ in $X$ has an associated carrier $H\times[-1,1]\subset X$, and we call the convex subcomplexes $H\times\{\pm1\}$ \emph{combinatorial hyperplanes}. For a convex subcomplex $K\subset X$ we let $\fg_K: X\to K$ denote the closest point projection to $K$.
	A collection $\fF$ of subcomplexes of $X$ is called a \emph{factor system} if it satisfies the following:
	\begin{enumerate}
		\item $X\in\fF$.
		\item Each $F\in\fF$ is a nonempty convex subcomplex of $X$.
		\item\label{item:countF} There exists $\Delta\geq 0$ such that for all $x\in X^{(0)}$, at most $\Delta$ elements of $\fF$ contain $x$.
		\item\label{item:hyp} Every nontrivial convex subcomplex parallel to a combinatorial hyperplane of $X$ is in $\fF$.
		\item\label{item:proj} There exists $\xi\geq0$ such that for any pair of subcomplexes $F,F'\in\fF$, either $\fg_F(F')\in\fF$ or diam$(\fg_F(F'))<\xi$.
	\end{enumerate}
\end{defn}

Our first theorem is as follows, which answers a question of Behrstock--Hagen--Sisto \cite[Question 8.13]{BehrstockHagenSisto17}.

\begin{thm}\label{thm:nofactor}
	There is a cubulated group $G\acts X$ such that $X$ does not have a factor system.
\end{thm}

In \cite[Theorem A]{HagenSusse20} Hagen--Susse provide three separate sufficient conditions for a cubulated group $G\acts X$ to admit a factor system: (1) the action is rotational, (2) it satisfies the weak finite height condition for hyperplanes, and (3) it satisfies the essential index condition together with the Noetherian intersection of conjugates condition on hyperplane stabilizers. Theorem \ref{thm:nofactor} gives the first known example of a cubulated group that fails all of these conditions (see Remark \ref{remk:3conditions} for more on this).
The example behind Theorem \ref{thm:nofactor} also contains pairs of hyperplanes that are $L$-well-separated but not $(L-1)$-well-separated for arbitrarily large $L$ (Remark \ref{remk:wellseparated}), which provides a negative answer to a question of Genevois \cite[Question 6.69 (first part)]{Genevois19}.

Associated to a CAT(0) cube complex $X$ is the \emph{contact graph} $\cC X$: the vertices are the hyperplanes of $X$, and edges correspond to pairs of hyperplanes whose carriers intersect (equivalently, pairs of hyperplanes that are not separated by a third hyperplane).
The contact graph is always a quasitree \cite{Hagen14}, so in particular it is hyperbolic.
Moreover, the contact graph is a key ingredient of the HHG structure that one usually builds for cubulated groups.
More precisely, if a cubulated group $G\acts X$ has a $G$-invariant factor system $\fF$, then one can build a HHG structure for $G$ by taking the contact graph $\cC F$ of each $F\in\fF$ and coning off certain subgraphs of $\cC F$ that correspond to smaller elements of $\fF$ -- see \cite{BehrstockHagenSisto17} for details.
The existence of a factor system for $G\acts X$ also implies that the induced action of $G$ on $\cC X$ is acylindrical \cite[Theorem D]{BehrstockHagenSisto17}.
(Recall that the action of a group $G$ on a metric space $(M,d)$ is \emph{acylindrical} if for all $\epsilon>0$ there exist $R,N>0$ such that $d(x,y)\geq R$ implies that there are at most $N$ elements $g\in G$ satisfying $d(x,gx),d(y,gy)<\epsilon$.)
The following theorem is therefore a strengthening of Theorem \ref{thm:nofactor}.

\begin{thm}\label{thm:notacyl}
	There is a cubulated group $G\acts X$ for which the induced action on the contact graph $\cC X$ is not acylindrical.
\end{thm}

This theorem is even more surprising in light of \cite[Theorem 1.1]{Genevois19b}, which implies that every cubulated group $G\acts X$ has a \emph{non-uniformly acylindrical} action on $\cC X$ (\emph{non-uniformly} meaning that ``at most $N$ elements'' is replaced by ``finitely many elements'' in the definition of acylindrical).

Briefly, the construction for Theorem \ref{thm:nofactor} is to take a free cocompact action of a group $\G$ on a product of trees $T_1\times T_2$ that contains an anti-torus, and then $\G$-equivariantly attach infinite strips to $T_1\times T_2$ along anti-tori axes.
The details are in Section \ref{sec:nofactor}.
The construction for Theorem \ref{thm:notacyl} builds on this by defining a certain HNN extension $\La=\G*_{\mathbb{Z}}$, and an action of $\La$ on a cube complex that splits as a tree of spaces, with vertex spaces being copies of $T_1\times T_2$ and edge spaces corresponding to the infinite strips described above.
The arguments for this are in Section \ref{sec:nonacyl}.
Although $\G$ admits a cubulation without a factor system, it is still a HHG because $\G\acts T_1\times T_2$ is another cubulation that does have a factor system.
On the other hand, we do not know whether our second group $\La$ admits a cubulation with a factor system, and we do not know whether $\La$ is a HHG.
In particular, the question of whether all cubulated groups are HHG's is still open \cite[Question A]{BehrstockHagenSisto19}.
One possible strategy is to find a cubulated group with no largest acylindrical action (see Definition \ref{defn:largest}), since all HHG's have a largest acylindrical action \cite{AbbottBehrstockDurham21}.
This does not work for the group $\La$ however, as we prove in Proposition \ref{prop:acyl} that $\La$ does have a largest acylindrical action.

\textbf{Acknowledgements:}\,
I am grateful for Mark Hagen's suggestion of considering anti-tori, which made my construction for Theorem \ref{thm:nofactor} more general.
Thanks go to Anthony Genevois for his comments, in particular regarding Remark \ref{remk:deltaK}.
And I thank the referee for their comments and corrections.
I am also thankful for the support of the Institut Henri Poincaré (UAR 839 CNRS-Sorbonne Université), and LabEx CARMIN (ANR-10-LABX-59-01).

\medskip
\section{Example with no factor system}\label{sec:nofactor}

Let $T_1$ and $T_2$ be locally finite trees, and let $\G$ be a group acting freely and cocompactly on $T_1\times T_2$.
Suppose that elements $g_1,g_2\in \G$ form an \emph{anti-torus}, meaning firstly that they translate non-trivially along intersecting axes $\ell_1\times\{p_2\}, \{p_1\}\times\ell_2\subset T_1\times T_2$ respectively (so $p_1\in\ell_1$ and $p_2\in\ell_2$), and secondly that no non-zero powers of $g_1$ and $g_2$ commute.
In addition, suppose that $g_1,g_2$ are not proper powers in $\G$.
The condition that no powers of $g_1$ and $g_2$ commute is equivalent to saying that the flat $\Pi=\ell_1\times\ell_2$ is not periodic.
We also note that the existence of an anti-torus implies that $\G$ is \emph{irreducible} \cite[Lemma 18]{JanzenWise09}, meaning it does not have a finite-index subgroup that splits as a product $\G_1\times\G_2$ with $\G_i$ acting trivially on $T_{3-i}$.
Examples of anti-tori were constructed by Wise \cite{Wise07}, Janzen--Wise \cite{JanzenWise09} and Rattaggi \cite{Rattaggi05}.
The smallest example is in \cite{JanzenWise09}, where $(T_1\times T_2)/\G$ consists of one vertex, four edges and four 2-cells.
See \cite{Caprace19} for more about anti-tori and irreducible lattices in products of trees.

Choose orientations for the edges in the finite quotient $(T_1\times T_2) /\G$, and label them with distinct letters from an alphabet $A$. Lift this labeling to $T_1\times T_2$. Each finite edge path in $T_1\times T_2$ or its quotient is thus labeled by some word $w$ on $A^\pm$, and we denote the length of $w$ by $|w|$.
The axes $\ell_1\times\{p_2\}, \{p_1\}\times\ell_2$ descend to loops in $(T_1\times T_2) /\G$ based at $\G\cdot(p_1,p_2)$; say these loops are labeled by words $w_1,w_2$ respectively.
Lifts of the $w_1$-loop to $T_1\times T_2$ will be referred to as \emph{$w_1$-geodesics} (equivalently, these are $\G$-translates of $\ell_1\times\{p_2\}$).

\begin{lem}\label{lem:w1geo}
	For any $n\geq 1$ there exists a $w_1$-geodesic whose intersection with $\Pi$ is a finite path of the form $\gamma\times\{y\}\subset \ell_1\times\ell_2$, with $p_1\in\gamma$ and $\gamma$ of length at least $n$.
\end{lem}
\begin{proof}
	For each $i\geq 1$, consider the rectangle in $\Pi$ with two sides labeled by $w_1^n$ and $w_2^i$ that meet at the vertex $(p_1,p_2)$, as shown in Figure \ref{fig:w1nw2i}.
	Note that the bottom side is a subpath of the axis $\ell_1\times\{p_2\}$, while the left-hand side is a subpath of the axis $\{p_1\}\times\ell_2$.
	Let $\alpha_i\times\{y_i\}$ denote the top side, and suppose that it is labeled by the word $v_i$.
	There are only finitely many words of length $|w_1^n|$, so $v_i=v_{i+j}$ for some $i,j\geq1$. 
	Since $\alpha_i\times\{y_i\}$ and $\alpha_{i+j}\times\{y_{i+j}\}$ are both labeled by $v_i$, the element $g_2^j$ maps $\alpha_i\times\{y_i\}$ to $\alpha_{i+j}\times\{y_{i+j}\}$.
	Moreover, $g_2^j$ preserves the axis $\{p_1\}\times\ell_2$, so it maps the rectangle shown in Figure \ref{fig:w1nw2i} to another rectangle in $\Pi$. Restricting to the bottom sides of the rectangles, we see that $g_2^j$ maps the subpath of $\ell_1\times\{p_2\}$ labeled $w_1^n$ to $\alpha_j\times\{y_j\}$, so $v_j=w_1^n$.
	
	\begin{figure}[H]
		\centering
		\scalebox{1}{
			\begin{tikzpicture}[auto,node distance=2cm,
				thick,every node/.style={circle,draw,fill,inner sep=0pt,minimum size=7pt},
				every loop/.style={min distance=2cm},
				hull/.style={draw=none},
				]
				\tikzstyle{label}=[draw=none,fill=none]
				\tikzstyle{a}=[isosceles triangle,sloped,allow upside down,shift={(0,-.05)},minimum size=3pt]
				
				\node at (0,0) {};
				\draw(0,0)-- node[a]{}(4,0);
				\draw(4,0)-- (4,3);
				\draw(0,0)-- node[a]{}(0,3);
				\draw(0,3)[blue,ultra thick]-- node[a]{}(4,3);

				\node[label] at (-.5,-.4) {$(p_1,p_2)$};
				\node[label] at (2,-.4) {$w_1^n$};
				\node[label] at (-.5,1.5) {$w_2^i$};
				\node[label] at (2,2.6) {$v_i$};
				\node[label,blue] at (2,3.4) {$\alpha_i\times\{y_i\}$};
				
			\end{tikzpicture}
		}
		\caption{Rectangle in $\Pi$ with two sides labeled by $w_1^n$ and $w_2^i$ that meet at the vertex $(p_1,p_2)$.}\label{fig:w1nw2i}
	\end{figure}
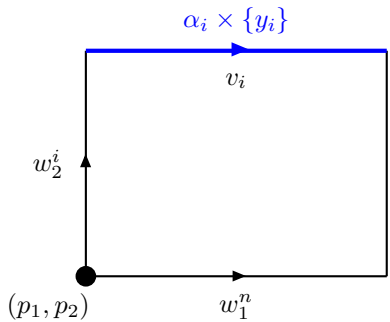
	
	The path $\alpha\times\{y\}:= \alpha_j\times\{y_j\}$ extends to a unique $w_1$-geodesic.
	Let $\gamma\times\{y\}\subset \ell_1\times\ell_2$ denote the intersection of this $w_1$-geodesic with $\Pi$.
	Since $p_1\in\alpha\subset\gamma$, and since $\alpha$ has length $|w_1^n|\geq n$, it remains to show that $\gamma$ is finite.
	Say that $p_1$ splits $\gamma$ into subpaths $\gamma_1,\gamma_2$, with $\alpha$ being an initial segment of $\gamma_2$.
	We will show that $\gamma_2$ is finite -- finiteness of $\gamma_1$ follows by a similar argument.
	
	For $k\geq n$, consider the rectangle in $\Pi$ with two sides labeled by $w_1^k$ and $w_2^j$ that meet at the vertex $(p_1,p_2)$, as shown in Figure \ref{fig:w1kw2j}. 
	Let $\beta_k\times\{y\}$ denote the top side.
	Note that $\alpha\times\{y\}$ is an initial segment of $\beta_k\times\{y\}$. Say the right-hand side is labeled by the word $v'_k$. The same argument we used earlier in the proof shows that $v'_k=w_2^j$ for some $k$. For this $k$, we then argue that $\gamma_2$ has length less that $|w_1^k|$. Indeed, otherwise $\gamma_2\times\{y\}$ would contain $\beta_k\times\{y\}$ as an initial segment, so $\beta_k\times\{y\}$ would be labeled by $w_1^k$, but then the labels on the rectangle would imply that $g_1^k$ and $g_2^j$ commute, contradicting the fact that $g_1,g_2$ form an anti-torus. 
	Thus $\gamma_2$ is finite, as required.
\end{proof}

	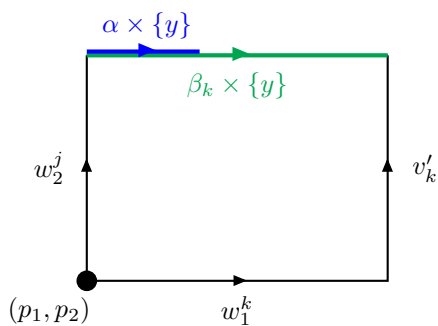
\begin{figure}[H]
	\centering
	\scalebox{1}{
		\begin{tikzpicture}[auto,node distance=2cm,
			thick,every node/.style={circle,draw,fill,inner sep=0pt,minimum size=7pt},
			every loop/.style={min distance=2cm},
			hull/.style={draw=none},
			]
			\tikzstyle{label}=[draw=none,fill=none]
			\tikzstyle{a}=[isosceles triangle,sloped,allow upside down,shift={(0,-.05)},minimum size=3pt]
			
			\node at (0,0) {};
			\draw(0,0)-- node[a]{}(4,0)-- node[a]{}(4,3);
			\draw(0,0)-- node[a]{}(0,3);
			\draw(0,3)-- (4,3);
			
			\draw[Green,ultra thick](0,3)--node[a]{}(4,3);
			\draw[blue,ultra thick](0,3.05)--node[a]{}(1.5,3.05);

			\node[label] at (-.5,-.4) {$(p_1,p_2)$};
			\node[label] at (2,-.4) {$w_1^k$};
			\node[label] at (-.5,1.5) {$w_2^j$};
			\node[label] at (4.5,1.5) {$v'_k$};
			\node[label,blue] at (.8,3.4) {$\alpha\times\{y\}$};
			\node[label,Green] at (2,2.6) {$\beta_k\times\{y\}$};
			
		\end{tikzpicture}
	}
	\caption{Rectangle in $\Pi$ with two sides labeled by $w_1^k$ and $w_2^j$ that meet at the vertex $(p_1,p_2)$.}\label{fig:w1kw2j}
\end{figure}

To construct a cubulation of $\G$ with no factor system we first take the quotient $(T_1\times T_2)/\G$, then we attach an annulus by gluing one boundary component along the edge loop labeled by $w_1$, and then we let $X$ be the universal cover. If the annulus is subdivided into $|w_1|$ squares then $X$ is a CAT(0) cube complex. Attaching the annulus to $(T_1\times T_2)/\G$ doesn't change the fundamental group, so $\G$ acts on $X$ by deck transformations.
The picture upstairs is that $X$ is obtained from $T_1\times T_2$ by attaching an infinite strip to each $w_1$-geodesic (and only one strip since $g_1$ is not a proper power).
We already remarked that $(T_1\times T_2)/\G$ has only four 2-cells for the example in \cite{JanzenWise09}, moreover the word $w_1$ has length two in this case, so the cube complex $X/\G$ would be a $\cV\cH$-complex consisting of just six 2-cells.

\begin{thm}\label{thm:Xnofactor}
	$X$ has no factor system.
\end{thm}
\begin{proof}
	Suppose for contradiction that $X$ has a factor system $\fF$.
	In $X$, there is an infinite strip glued to each $w_1$-geodesic, and there is a hyperplane that runs along the middle of each infinite strip (shown as dotted red lines in Figure \ref{fig:factor}). Hence each $w_1$-geodesic is a combinatorial hyperplane in $X$, and is an element of $\fF$ by Definition \ref{defn:factor}\ref{item:hyp}.
	In particular $F:=\ell_1\times\{p_2\}\in\fF$.
	
	Choose an integer $n\geq1$ and apply Lemma \ref{lem:w1geo}. This provides us with a $w_1$-geodesic $F'$ whose intersection with $\Pi$ is a finite path of the form $\gamma\times\{y\}\subset \ell_1\times\ell_2=\Pi$, with $p_1\in\gamma$ and $\gamma$ of length at least $n$.
	The projection $\fg_F:X\to F$ maps $\gamma\times\{y\}$ to $\gamma\times\{p_2\}$, and it maps the rest of $F'$ to the endpoints of $\gamma\times\{p_2\}$. By Definition \ref{defn:factor}\ref{item:proj}, the image $\fg_F(F')=\gamma\times\{p_2\}$ is in $\fF$ for sufficiently large $n$. But we have $(p_1,p_2)\in\fg_F(F')$ for all $n$, so we contradict Definition \ref{defn:factor}\ref{item:countF}. 
\end{proof}
	\begin{figure}[H]
	\centering
	\scalebox{1}{
		\begin{tikzpicture}[auto,node distance=2cm,
			thick,every node/.style={circle,draw,fill,inner sep=0pt,minimum size=7pt},
			every loop/.style={min distance=2cm},
			hull/.style={draw=none},
			]
			\tikzstyle{label}=[draw=none,fill=none]
			\tikzstyle{a}=[isosceles triangle,sloped,allow upside down,shift={(0,-.05)},minimum size=3pt]
			
			\begin{scope}[scale=.6]
			\draw (0,0) rectangle (10,8);
			\draw (-5,0) grid (15,-1);
			\draw (-5.6,-1)--(15.6,-1);
			\draw[blue, ultra thick] (-5.6,0)--(15.6,0);
			\draw[dashed, red, ultra thick] (-5.6,-.5)--(15.6,-.5);

			\draw (0,8) grid (10,9);
			\foreach\x in {1,2,3,4,5} \draw (-\x,8+.3*\x)--(-\x,9+.3*\x)(10+\x,8+.3*\x)--(10+\x,9+.3*\x);
			\draw (-5.6,9+.3*5.6)--(0,9)--(10,9)--(15.6,9+.3*5.6);
			\draw[blue, ultra thick] (-5.6,8+.3*5.6)--(0,8)--(10,8)--(15.6,8+.3*5.6);
			\draw[dashed, red, ultra thick](-5.6,8.5+.3*5.6)--(0,8.5)--(10,8.5)--(15.6,8.5+.3*5.6);
			
			\draw[Green, ultra thick](0,.1)--(10,.1);
			\draw[Green, ultra thick](0,7.9)--(10,7.9);
			\node at (3,0){};
			
			\node[label,blue] at (-4,.5) {$F=\ell_1\times\{p_2\}$};
			\node[label,blue] at (-5,9) {$F'$};
			\node[label,Green] at (7,.6) {$\fg_F(F')=\gamma\times\{p_2\}$};
			\node[label,Green] at (5,7.4) {$\gamma\times\{y\}$};
			\node[label] at (2,.6) {$(p_1,p_2)$};
			\node[label,red] at (-6.2,-.5) {$H$};
			\node[label,red] at (-6.2,10.3) {$H'$};
			\end{scope}			
		\end{tikzpicture}
	}
	\caption{The $w_1$-geodesics $F$ and $F'$ with their attached strips.}\label{fig:factor}
\end{figure}
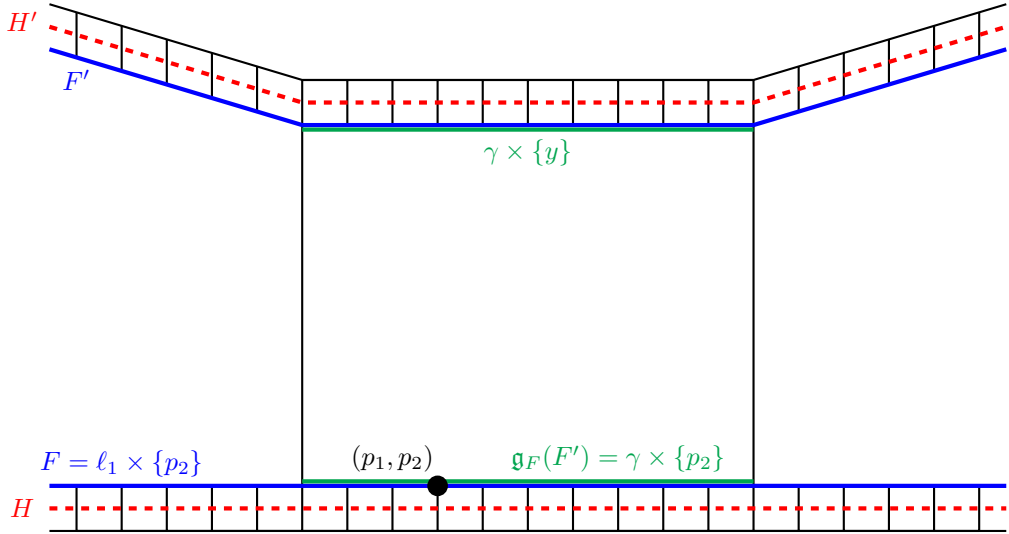

\begin{remk}\label{remk:wellseparated}
	If $H,H'$ are the hyperplanes that run along the strips glued to $F,F'$ from the above proof (see Figure \ref{fig:factor}), then the hyperplanes that are transverse to both $H$ and $H'$ are precisely the hyperplanes that cross the path $\gamma\times\{y\}$. Moreover, this collection of hyperplanes has no facing triples, so if $\gamma$ has length $L$ then $H,H'$ are $L$-well-separated but not $(L-1)$-well-separated. We can choose $H,H'$ so that $L$ is arbitrarily large, so this provides a negative answer to a question of Genevois \cite[Question 6.69 (first part)]{Genevois19}.
\end{remk}

\begin{remk}\label{remk:3conditions}
	In \cite[Theorem A]{HagenSusse20} Hagen--Susse provide three separate sufficient conditions for a cubulated group to admit a factor system.
	Since $X$ has no factor system, we know that $\Gamma\acts X$ does not satisfy any of these three conditions. We describe what these conditions are below, and outline some direct arguments for why they fail for $\G\acts X$:
	\begin{enumerate}
		\item A cubulated group $G\acts Z$ is \emph{rotational} if for each hyperplane $B$ there is a finite-index subgroup $K_B<\rm{Stab}_G(B)$ such that, for any hyperplane $A$ disjoint from $B$ and all $k\in K_B$, the carriers of $A$ and $kA$ are either equal or disjoint.
		
		For our cubulated group $\Gamma\acts X$, we can consider $B=H$ to be the hyperplane shown in Figure \ref{fig:factor}, and we can show that for any $1\neq k\in\rm{Stab}_\Gamma(H)$ there is a hyperplane $A$ disjoint from $H$ such that the carriers of $A$ and $kA$ are distinct but not disjoint.
		Indeed, for $1\neq k\in\rm{Stab}_\Gamma(H)=\langle g_1\rangle$ we know that $k(\{p_1\}\times\ell_2)\cap\Pi$ is a finite path (else $k$ would commute with some power of $g_2$ by a similar argument to the proof of Lemma \ref{lem:w1geo}, contradicting the fact that $g_1, g_2$ form an anti-torus), so there must be a vertex $x\in\ell_2$ incident to an edge $e\subset\ell_2$ such that $k(p_1,x)\in \ell_1\times\{x\}$ but $k(\{p_1\}\times e)\nsubseteq\Pi$.
		With $A$ being the hyperplane dual to $\{p_1\}\times e$, the intersection of the carriers of $A$ and $kA$ is $T_1\times\{x\}$, which is a proper subset of the carrier of $A$ (it is one of the combinatorial hyperplanes of $A$).
		
		\item A cubulated group $G\acts Z$ satisfies the \emph{weak finite height condition for hyperplanes} if the following holds for each hyperplane $A$ and its stabilizer $K=\rm{Stab}_G(A)$: If $\{g_i\}\subset G$ is an infinite set such that $K\cap\bigcap_{i\in J}K^{g_i}$ is infinite for all finite $J\subset I$, then there exist distinct $g_i,g_j$ so that $K\cap K^{g_i}=K\cap K^{g_j}$.
		
		This condition fails for our cubulated group $\Gamma\acts X$ (in fact it also fails for $\Gamma\acts T_1\times T_2$) by considering an edge $e\subset\ell_2$ and taking $A$ to the hyperplane dual to $\{p_1\}\times e$.
		Then the stabilizer $K=\rm{Stab}_\Gamma(A)$ is just the stabilizer of $e$ with respect to the action $\Gamma\acts T_2$.
		For any $i\geq1$ we know that $g_1^i(\{p_1\}\times \ell_2)\cap\Pi$ is finite (as in case (1)), or equivalently $g_1^i\ell_2\cap\ell_2\subset T_2$ is finite.
		The element $g_2$ translates along $\ell_2$, so for any power $g_2^j$ the conjugate $K^{g_2^j}$ is the $\Gamma$-stabilizer of the edge $g_2^je\subset\ell_2$.
		Thus $g_1^i\notin K^{g_2^j}$ for all sufficiently large $j\geq1$.
		Since $T_2$ is locally finite, $K$ and $K^{g_2^j}$ are commensurable in $\Gamma$ for any $j$, and $\langle g_1\rangle \cap K^{g_2^j}$ is infinite since $g_1$ fixes the vertex $p_2\in T_2$.
		Hence, we may construct increasing sequences of positive integers $(i_k)$ and $(j_k)$ such that $g_1^{i_k}$ is not in $K^{g_2^{j_k}}$ but $g_1^{i_{k+1}}$ is.
		Therefore, 
		$$K\cap K^{g_2^{j_1}}\supsetneq K\cap K^{g_2^{j_2}}\supsetneq K\cap K^{g_2^{j_3}}\supsetneq\dots$$
		is a strictly descending chain of commensurable (hence infinite) subgroups of $\Gamma$, which means the weak finite height condition for hyperplanes does not hold for $\Gamma\acts X$.
		
		\item The third condition has two parts. A cubulated group $G\acts Z$ satisfies the \emph{essential index condition} if there is a constant $\zeta$ such that for any $F\in\fF$ (where $\fF$ is the smallest collection of convex subcomplexes of $Z$ that contains $Z$, contains all combinatorial hyperplanes, and is closed under closest point projection) the $G$-stabilizer of $F$ has index at most $\zeta$ in the $G$-stabilizer of the essential core of $F$.
		A cubulated group $G\acts Z$ satisfies the \emph{Noetherian intersection of conjugates (NIC) condition on hyperplane stabilizers} if the following holds for each hyperplane stabilizer $K$: Given $\{g_i\}\subset G$ so that $K_n=K\cap\bigcap_{i=0}^n K^{g_i}$ is infinite for all $n$, there exists $l$ so that $K_n$ and $K_l$ are commensurable for $n\geq l$.
		
		Our cubulated group $\G\acts X$ satisfies the NIC condition for hyperplane stabilizers (because all hyperplane stabilizers are either cyclic or commensurated) but fails the essential index condition as follows.
		Suppose the the geodesic $\ell_1\times\{p_2\}$ (from Figure \ref{fig:factor}) crosses the hyperplanes $H_1,H_2,H_3,\dots$ respectively when starting at $(p_1,p_2)$ and moving in the direction of translation of $g_1$.
		Let $F_i$ be the combinatorial hyperplane of $H_i$ that is on the same side of $H_i$ as $(p_1,p_2)$.
		Let $\fF$ be as described above for our cubulated group $\Gamma\acts X$, and consider the subcomplexes $\fg_{F_1}(F_i)\in\fF$.
		Given an integer $n\geq1$, as in the proof of Theorem \ref{thm:Xnofactor} we can apply Lemma \ref{lem:w1geo} to obtain a $w_1$-geodesic $F'$ whose intersection with $\Pi$ is a finite path of the form $\gamma\times\{y\}\subset \ell_1\times\ell_2=\Pi$, with $p_1\in\gamma$ and $\gamma$ of length at least $n$.
		Moreover, it follows from the proof of Lemma \ref{lem:w1geo} that we can take $(p_1,y)=g_2^k(p_1,p_2)$ for some $k\neq0$, and we may assume that the subpath of $\gamma\times\{y\}$ that starts at $(p_1,y)$ and moves in the direction of translation of $g_2^kg_1g_2^{-k}$ has length at least $n$.
		Let $e,e'$ be the edges in $F_1$ incident at vertices $(p_1,p_2),(p_1,y)$ respectively that cross the hyperplanes $H$, $H'$ respectively (again from Figure \ref{fig:factor}).
		Note that $g_2^ke=e'$.
		One can show that $e\subset\fg_{F_1}(F_i)$ for all $i$.
		On the other hand, $e'\subset\fg_{F_1}(F_i)$ for $1\leq i\leq n$ but for only finitely many $i$ in total.
		One can then argue that $g_2^k$ is in the $\Gamma$-stabilizer of $\fg_{F_1}(F_i)$ for $1\leq i\leq n$ but for only finitely many $i$ in total.
		This can be done for any $n\geq1$, so the $\Gamma$-stabilizers of the $\fg_{F_1}(F_i)$ is a descending sequence of subgroups that never terminates.
		Meanwhile, all the $\fg_{F_1}(F_i)$ have essential core $\{p_1\}\times T_2$, so the essential index condition fails.
		
	\end{enumerate}
\end{remk}

\medskip
\section{Example with non-acylindrical action on the contact graph}\label{sec:nonacyl}

We now construct a free cocompact action of a group $\La$ on a CAT(0) cube complex $Y$, such that the induced action on the contact graph $\cC Y$ is not acylindrical. We start with the action of $\G$ on $T_1\times T_2$ from Section \ref{sec:nofactor}, with elements $g_1,g_2\in\G$ forming an anti-torus. We retain all the notation from Section \ref{sec:nofactor}, so in particular $g_1$ translates along an axis $\ell_1\times\{p_2\}\subset T_1\times T_2$ which descends to a loop in $(T_1\times T_2)/\G$ labeled by $w_1$.
Next we attach an annulus to $(T_1\times T_2)/\G$ by gluing both its boundary components (with matching orientations) along the loop labeled by $w_1$.
We make this a non-positively curved cube complex by subdividing the annulus into $|w_1|$ squares, and we let $Y$ be the universal cover.
Unlike for the construction of $X$ in Section \ref{sec:nofactor}, we have glued \emph{both} boundary components of the annulus to $(T_1\times T_2)/\G$, so the gluing changes the fundamental group from $\G$ to the HNN extension
\begin{equation}\label{Lapres}
\La:=\G *_{\langle g_1\rangle}=\langle \G,t\mid tg_1 t^{-1}=g_1\rangle.
\end{equation}
And $\La$ acts on $Y$ by deck transformations.
Observe that $Y$ has the structure of a tree of spaces, where the vertex spaces are copies of $T_1\times T_2$, and the edge spaces are infinite strips.
The edge labeling of $(T_1\times T_2)/\G$ induces an edge labeling of $Y/\La$ (apart from the edges that cross the annulus) and we can lift this to an edge labeling of $Y$.
As in Section \ref{sec:nofactor}, lifts of the $w_1$-loop in $Y/\La$ to $Y$ will be referred to as \emph{$w_1$-geodesics}.
Each $w_1$-geodesic in $Y$ is attached to two edge spaces since the $w_1$-loop in $Y/\La$ is attached to both boundary components of the annulus.

\begin{thm}\label{thm:cCYnotacyl}
	The action of $\La$ on the contact graph $\cC Y$ is not acylindrical.
\end{thm}
\begin{proof}
	We will show that for any $R,N>0$ there exist $H,H'\in\cC Y$ such that $d_{\cC Y}(H,H')\geq R$ and there are more than $N$ elements $g\in \La$ satisfying $d_{\cC Y}(H,gH), d_{\cC Y}(H',gH')\leq2$.
	
Consider a vertex space in $Y$, and identify it with $T_1\times T_2$. Given an integer $n\geq 1$, we can choose $w_1$-geodesics $F,F'$ as in the proof of Theorem \ref{thm:Xnofactor} such that the projection $\fg_F(F')$ is a finite path of length at least $n$ that contains the vertex $(p_1,p_2)$.
Now take one of the edge spaces in $Y$ glued to $F'$, and let $F''$ be the geodesic on the other side of the edge space.
$F''$ is in a different vertex space, but it is again a $w_1$-geodesic, so it contains vertices in the orbit $\La\cdot(p_1,p_2)$. Moreover, the spacing between these vertices is at most $|w_1|$, so we can choose $g_n\in\La$ such that $g_n(p_1,p_2)$ lies on $F''$ but is shifted to the right relative to $(p_1,p_2)$ by an integer $0<r\leq|w_1|$, as shown in Figure \ref{fig:g}.

	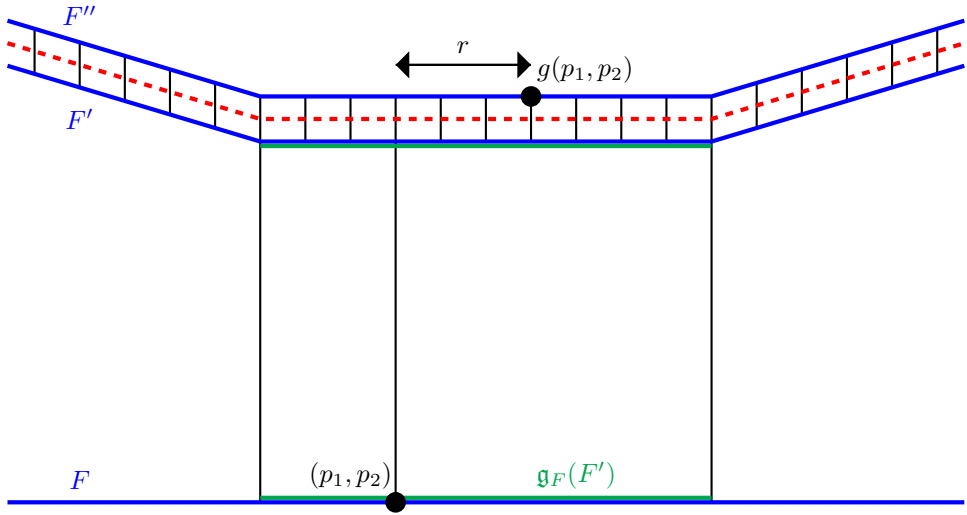
\begin{figure}[H]
	\centering
	\scalebox{1}{
		\begin{tikzpicture}[auto,node distance=2cm,
			thick,every node/.style={circle,draw,fill,inner sep=0pt,minimum size=7pt},
			every loop/.style={min distance=2cm},
			hull/.style={draw=none},
			]
			\tikzstyle{label}=[draw=none,fill=none]
			\tikzstyle{a}=[isosceles triangle,sloped,allow upside down,shift={(0,-.05)},minimum size=3pt]
			
			\begin{scope}[scale=.6]
				\draw (0,0) rectangle (10,8);
				\draw (3,0)--(3,8);
				\draw[blue, ultra thick] (-5.6,0)--(15.6,0);

				\draw (0,8) grid (10,9);
				\foreach\x in {1,2,3,4,5} \draw (-\x,8+.3*\x)--(-\x,9+.3*\x)(10+\x,8+.3*\x)--(10+\x,9+.3*\x);
				\draw[blue, ultra thick] (-5.6,9+.3*5.6)--(0,9)--(10,9)--(15.6,9+.3*5.6);
				\draw[blue, ultra thick] (-5.6,8+.3*5.6)--(0,8)--(10,8)--(15.6,8+.3*5.6);
				\draw[dashed, red, ultra thick](-5.6,8.5+.3*5.6)--(0,8.5)--(10,8.5)--(15.6,8.5+.3*5.6);
				
				\draw[Green, ultra thick](0,.1)--(10,.1);
				\draw[Green, ultra thick](0,7.9)--(10,7.9);
				\node at (3,0){};
				\node at (6,9){};
				\draw[-triangle 90] (3,9.7)--(6,9.7);
				\draw[-triangle 90] (6,9.7)--(3,9.7);
				
				\node[label,blue] at (-4,.5) {$F$};
				\node[label,blue] at (-4,8.5) {$F'$};
				
				\node[label,blue] at (-4,10.8) {$F''$};
				\node[label,Green] at (7,.6) {$\fg_F(F')$};
				\node[label] at (2,.6) {$(p_1,p_2)$};
				\node[label] at (7.2,9.6) {$g_n(p_1,p_2)$};
				\node[label] at (4.5,10.1) {$r$};
			\end{scope}			
		\end{tikzpicture}
	}
	\caption{The positioning of $g_n(p_1,p_2)$.}\label{fig:g}
\end{figure}

Note that $g_nF=F''$. Furthermore, applying powers of $g_n$ to $F$ and $F'$ produces a staircase-like picture as shown in Figure \ref{fig:staircase}, where each step has depth $r$.
Each pair $g_n^iF,g_n^iF'$ lies in a different vertex space of $Y$.
If $H$ is the hyperplane that runs along the edge space between $g_n^{-1}F'$ and $F$, then the hyperplanes $g_n^i H$ run along a sequence of edge spaces that connect the aforementioned sequence of vertex spaces.

	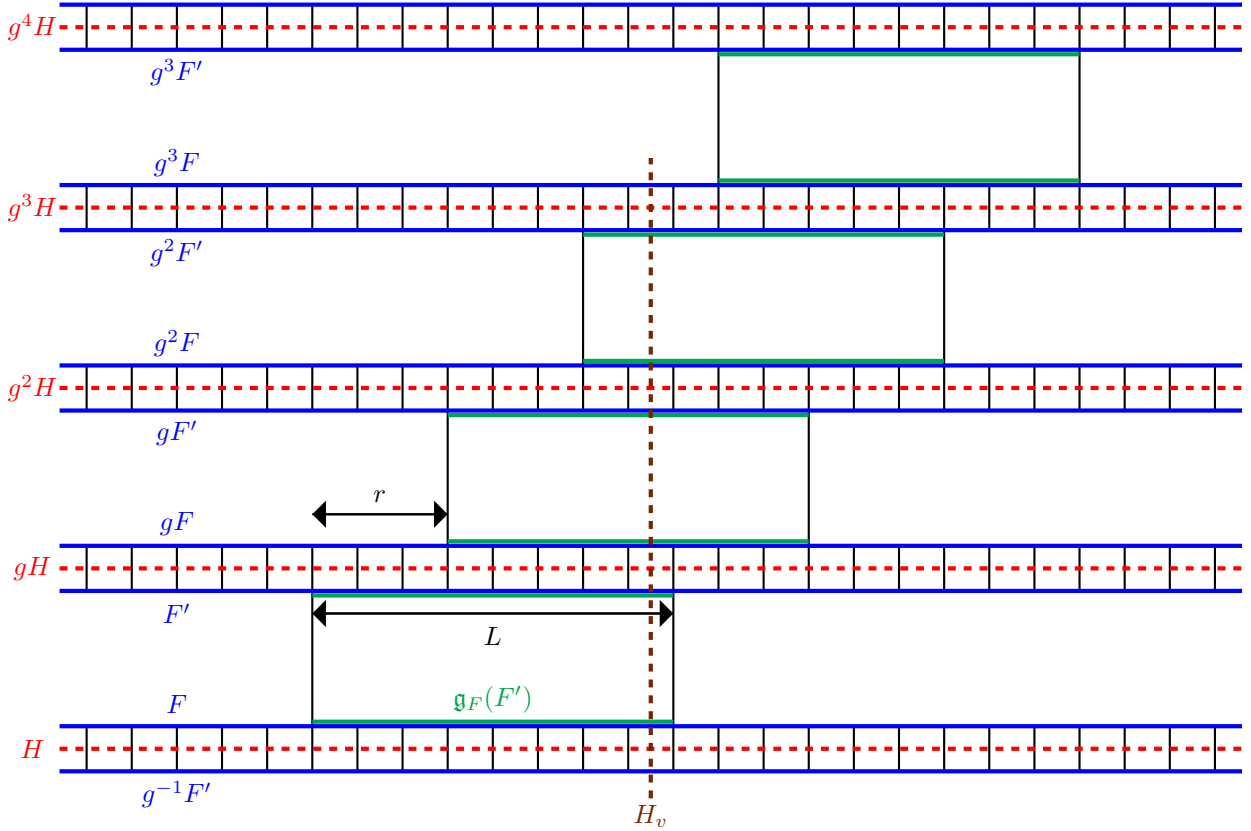
\begin{figure}[H]
	\centering
	\scalebox{1}{
		\begin{tikzpicture}[auto,node distance=2cm,
			thick,every node/.style={circle,draw,fill,inner sep=0pt,minimum size=7pt},
			every loop/.style={min distance=2cm},
			hull/.style={draw=none},
			]
			\tikzstyle{label}=[draw=none,fill=none]
			\tikzstyle{a}=[isosceles triangle,sloped,allow upside down,shift={(0,-.05)},minimum size=3pt]
						
			\begin{scope}[scale=.6]
				
				\draw (-5,-1) grid (20,0);
				\draw[dashed, red, ultra thick](-5.6,-.5)--(20.6,-.5);
				\draw[blue, ultra thick] (-5.6,-1)--(20.6,-1);
				\foreach\x in {0,1,2,3}
				{\draw (3*\x,4*\x) rectangle (3*\x+8,4*\x+3);
				\draw (-5,4*\x+3) grid (20,4*\x+4);
				\draw[dashed, red, ultra thick](-5.6,4*\x+3.5)--(20.6,4*\x+3.5);
				\draw[blue, ultra thick] (-5.6,4*\x)--(20.6,4*\x)(-5.6,4*\x+3)--(20.6,4*\x+3);
				\draw[Green, ultra thick](3*\x,4*\x+.1)--(3*\x+8,4*\x+.1)(3*\x,4*\x+2.9)--(3*\x+8,4*\x+2.9);
			}
				
			\draw[blue, ultra thick] (-5.6,4*4)--(20.6,4*4);
			
			\node[label,blue] at (-3,.5) {$F$};
			\node[label,blue] at (-3,2.5) {$F'$};
			\node[label,blue] at (-3,-1.5) {$g_n^{-1} F'$};
			\node[label,Green] at (4,.6) {$\fg_F(F')$};
			\node[label,red] at (-6.2,-.5) {$H$};
			\node[label,red] at (-6.2,3.5) {$g_nH$};
			
			\node[label,blue] at (-3,4.5) {$g_nF$};
			\node[label,blue] at (-3,6.5) {$g_nF'$};
			\node[label,red] at (-6.2,7.5) {$g_n^2H$};
			\node[label,red] at (-6.2,11.5) {$g_n^3H$};
			\node[label,red] at (-6.2,15.5) {$g_n^4H$};
			
			\foreach\x in {2,3}
			{\node[label,blue] at (-3,4*\x+.5) {$g_n^\x F$};
				\node[label,blue] at (-3,4*\x+2.5) {$g_n^\x F'$};
			}
		
		\draw[-triangle 90] (0,2.5)--(8,2.5);
		\draw[-triangle 90] (8,2.5)--(0,2.5);
		\node[label] at (4,2) {$L$};
		\draw[Brown, dashed, ultra thick] (7.5,-1.6)--(7.5,12.6);
		\node[label,Brown] at (7.5,-2) {$H_v$};
		\draw[-triangle 90] (0,4.7)--(3,4.7);
		\draw[-triangle 90] (3,4.7)--(0,4.7);
		\node[label] at (1.5,5.1) {$r$};	
		
			\end{scope}							
		\end{tikzpicture}
	}
	\caption{The arrangement of the $w_1$-geodesics $g_n^i F$ and $g_n^i F'$.}\label{fig:staircase}
\end{figure}

Suppose that the path $\fg_F(F')$ has length $L$ (remembering that $L\geq n$).
If a hyperplane intersects more than one of the hyperplanes $g_n^i H$ then it must cross some $\langle g_n\rangle$-translate of the projection $\fg_F(F')$.
The hyperplane $H_v$ that is dual to the last edge of $\fg_F(F')$ intersects exactly $M:=\ceil{L/r}+1$ of the hyperplanes $g_n^i H$, and no other hyperplane intersects more of them.
In particular, for $1\leq i< M$, $H_v$ intersects $H$ and $g_n^iH$, so the distance between $H$ and $g_n^iH$ in the contact graph $\cC Y$ is
\begin{equation}\label{dcC2}
d_{\cC Y}(H,g_n^i H)=2.
\end{equation}
On the other hand, $d_{\cC Y}(H,g_n^p H)\to\infty$ as $p\to\infty$. Indeed, suppose the geodesic in $\cC Y$ from $H$ to $g_n^p H$ consists of hyperplanes $H=H_0,H_1,...,H_d=g_n^p H$. We know that $H,g_n^pH$ are separated by the hyperplanes $g_n^i H$ for each $0< i< p$, so each of these $g_n^i H$ either equals one of the $H_j$ or intersects one of the $H_j$.
But we know from earlier that each $H_j$ intersects at most $M$ of the $g_n^i H$, hence
$$d=d_{\cC Y}(H,g_n^p H)\geq\frac{p}{M}.$$
Putting $H'=g_n^p H$, we have $d_{\cC Y}(H,H')\geq R$ provided $p\geq MR$. But (\ref{dcC2}) implies that $H$ and $H'$ are both moved at most distance 2 in $\cC Y$ by the elements $\{1,g_n,g_n^2,...,g_n^{M-1}\}\subset\La$. And $M=\ceil{L/r}+1> n/r$, so if we choose $n\geq rN$ then we get more than $N$ elements $g\in \La$ satisfying $d_{\cC Y}(H,gH), d_{\cC Y}(H',gH')\leq2$, as required. So we conclude that the action of $\La$ on $\cC Y$ is not acylindrical.
\end{proof}

\begin{remk}
	The staircase in Figure \ref{fig:staircase} is not a staircase as defined in \cite{HagenSusse20} (visually speaking, the former looks like it has empty space below the staircase whereas the latter does not).
	However, they are both obstructions to the existence of a factor system.
	Indeed, the existence of staircases as in Figure \ref{fig:staircase} for arbitrarily large ratios $L/r$ is the key to proving Theorem \ref{thm:cCYnotacyl}, which in turn implies that $Y$ has no factor system.
	Meanwhile, the existence of just a single convex staircase in the sense of \cite{HagenSusse20} rules out the possibility of a factor system.
	It remains an open question whether any cubulation of a group contains a convex staircase in the sense of \cite{HagenSusse20}.
\end{remk}

\begin{remk}\label{remk:deltaK}
In \cite{Genevois19} Genevois defines a metric $\delta_K$ for a CAT(0) cube complex $X$ as the maximal number of pairwise $K$-well-separated hyperplanes separating two given vertices. The space $(X,\delta_K)$ is hyperbolic for all $K$, and it is quasi-isometric to the contact graph for $K=0$. 
Moreover, if $G\acts X$ is a cubulated group, then the induced action on $(X,\delta_K)$ is non-uniformly acylindrical for all $K$.
However, for our cubulated group $\La\acts Y$, the action of $\La$ on $(Y,\delta_K)$ is not acylindrical for any $K$.
The argument is similar to that in the proof of Theorem \ref{thm:cCYnotacyl}: you can define the element $g_n\in\La$ in the same way, and show that $g_n$ acts loxodromically on $(Y,\delta_K)$, and you can exhibit points on the axis of $g_n$ that are arbitrarily far apart but are moved at most distance 2 by many powers of $g_n$ -- with the number of such powers tending to infinity as $n\to\infty$.
\end{remk}

As discussed in the introduction, we do not know whether $\La$ is a HHG.
A possible strategy to prove that $\La$ is not a HHG would be to show that $\La$ has no largest acylindrical action (definition below), because all HHG's possess a largest acylindrical action \cite{AbbottBehrstockDurham21}.

\begin{defn}\label{defn:largest}
	Let $G$ be a group that acts on metric spaces $R$ and $S$. We say that $G\acts R$ is \emph{dominated} by $G\acts S$, written $G\acts R\preceq G\acts S$, if there exist $r\in R$, $s\in S$ and a constant $C$ such that
	$$d_R(r,gr)\leq C d_S(s,gs)+C$$
	for all $g\in G$.
	The actions $G\acts R$ and $G\acts S$ are \emph{equivalent} if $G\acts R\preceq G\acts S$ and $G\acts S\preceq G\acts R$.
	We denote the equivalence class by $[G\acts R]$.
	The relation $\preceq$ defines a partial order on the set of equivalence classes of actions of $G$ on metric spaces.
	The \emph{largest acylindrical action} of $G$ (if it exists) is the largest element of the set of equivalence classes of cobounded acylindrical actions of $G$ on hyperbolic metric spaces. 
\end{defn}

Alas, we show below in Proposition \ref{prop:acyl} that $\La$ does have a largest acylindrical action, which is defined as follows.
Let $T$ be the Bass-Serre tree for the splitting $\La=\G*_{\langle g_1\rangle}$. We say that edges $e_1,e_2\in ET$ are \emph{equivalent} if the stabilizers $\La_{e_1},\La_{e_2}$ are commensurable. Each equivalence class defines a subtree of $T$ called a \emph{cylinder}.
The \emph{tree of cylinders} $T_c$ is the bipartite tree with vertex set $V_0T_c \sqcup V_1T_c$, where $V_0T_c$ are the vertices of $T$ and $V_1T_c$ is the set of cylinders.
The edges of $T_c$ are of the form $(v, C)$ where $v$ is a vertex in $T$ that lies in the cylinder $C \subset T$.
The action of $\La$ on $T$ induces an action on $T_c$.

To help us prove Proposition \ref{prop:acyl} we will use the following lemma, which is equivalent to \cite[Corollary 4.14]{AbbottBalasubramanyaOsin19}.

\begin{lem}\label{lem:dominate}
	Let $G$ act cocompactly on a connected graph $\Delta$ and let $G\acts R$ be another cobounded action on a metric space. If the vertex stabilizers of $\Delta$ have bounded orbits in $R$, then $G\acts R\preceq G\acts \Delta$.
\end{lem}

\begin{prop}\label{prop:acyl}
	The largest acylindrical action of $\La$ is its action on the tree of cylinders $\La\acts T_c$.
\end{prop}
\begin{proof}
First we show that the action of $\La$ on $T_c$ is acylindrical.
It suffices to show that only the identity element of $\La$ fixes a path in $T_c$ of the form $v_1,C_1,u,C_2,v_2$, where $u,v_i\in V_0 T_c$ and $C_i\in V_1 T_c$. Indeed if $g\in\Lambda$ fixes such a path, then it must also fix any pair of edges $e_1,e_2\in ET$ that lie on the geodesics $[u,v_1]$ and $[u,v_2]$ respectively.
It follows that $e_1,e_2$ lie in the cylinders $C_1,C_2$, and so the stabilizers $\La_{e_1},\La_{e_2}$ are not commensurable. But these stabilizers are infinite cyclic (they are conjugates of $\langle g_1\rangle$), hence they have trivial intersection, so $g\in\La_{e_1}\cap\La_{e_2}$ is trivial.

The action $\La\acts T_c$ is clearly cobounded and $T_c$ is obviously hyperbolic, so it remains to show that $\La\acts T_c$ dominates all other cobounded acylindrical actions of $\La$ on hyperbolic spaces. 
Let $\La\acts R$ be such an action.
By Lemma \ref{lem:dominate}, it suffices to show that the vertex stabilizers of $T_c$ have bounded orbits in $R$.

First consider $v\in V_0 T_c$. If the stabilizer $\La_v$ does not have bounded orbits in $R$ then there must exist a loxodromic element $g\in\La_v$ by \cite[Theorem 1.1]{Osin16}, and $g$ must be contained in a virtually cyclic hyperbolically embedded subgroup of $\La_v$ by \cite[Theorem 1.4]{Osin16}. It then follows from \cite[Theorem 1]{Sisto16} that $g$ is Morse in $\La_v$. But this is impossible since $\La_v\cong\G$ is quasi-isometric to a product of trees.

Now consider $C\in V_1 T_c$. Without loss of generality we may assume that $C$ contains the edge that corresponds to the subgroup $\langle g_1\rangle <\G<\La$. We know that $\La_C$ contains the element $g_1$ as well as the stable letter $t$ from the presentation (\ref{Lapres}), so $\La_C$ is not virtually cyclic. 
If $\La_C$ does not have bounded orbits in $R$ then there must exist a loxodromic element $g\in\La_C$ by \cite[Theorem 1.1]{Osin16}.
Fix a point $r\in R$.
The element $g_1$ lies in the $V_0 T_c$-vertex stabilizer $\G<\La$, so $g_1$ is elliptic in $R$ by the previous paragraph, and the orbit $\langle g_1\rangle\cdot r$ lies in some $\epsilon$-ball about $r$.
By definition of $C$, for any integer $k$ the subgroups $\langle g_1\rangle$ and $g^k\langle g_1\rangle g^{-k}$ are commensurable, so have infinite intersection; and any element $h$ of this intersection satisfies $d_R(r,hr),d_R(g^k r,hg^k r)<\epsilon$. But $d(r,g^kr)\to\infty$ as $k\to\infty$ since $g$ is loxodromic, which contradicts the acylindricity of the action $\La\acts R$.
\end{proof}

\bibliographystyle{alpha}
\bibliography{Ref}

\end{document}